\documentclass[11pt,a4paper]{article}
\usepackage[affil-it]{authblk}
\usepackage[utf8]{inputenc}
\usepackage{amsmath, amssymb, mathrsfs, amsthm, amsfonts}
\usepackage{latexsym}
\usepackage{hyperref}
\usepackage{a4wide}
\usepackage{color}
\usepackage{graphicx}
\usepackage{tikz}
\usepackage{tikz-cd}

\usetikzlibrary{arrows}

\newtheorem{theorem}{Theorem}[section]

\newtheorem{lemma}[theorem]{Lemma}
\newtheorem{proposition}[theorem]{Proposition}

\theoremstyle{definition}
\newtheorem{definition}[theorem]{Definition}
\theoremstyle{remark}
\newtheorem{example}{Example}

\title{Random Walk of a Cat in a Building}

\author{Hery Randriamaro
\thanks{This research was funded by my mother \\
Lot II B 32 bis Faravohitra, 101 Antananarivo, Madagascar \\
e-mail: \texttt{hery.randriamaro@gmail.com}}}

\begin{document}

\maketitle

\begin{abstract}
\noindent One usually thinks of a cat moving from one room to another in an apartment as random walk model. Imagine now that it also has the possibility to go from one apartment to another by crossing some corridors. That yields a new probabilistic model for which each corridor connects the entrance rooms of several apartments. This article shows that the determinants of the stochastic and the exponential distance matrices of that model have a nice factorization. Two examples involving indirectly acyclic digraphs and hyperplane arrangements are provided.   

\bigskip 

\noindent \textsl{Keywords}: Random Walk, Stochastic Matrix, Distance Function, Determinant

\smallskip

\noindent \textsl{MSC Number}: 05B20, 15A15, 60C05, 60J10
\end{abstract}

\section{Introduction}

\noindent A random walk is a stochastic model describing the probability of random steps on some mathematical space. We use a connected digraph $\mathrm{G} := (\mathrm{V},\mathrm{E})$, where $\mathrm{V}$ is the set of vertices and $\mathrm{E}$ of edges, to describe our model. For every pair $(A,B)$ of vertices, there is a vertex sequence $(A = A_1, A_2, \dots, A_k = B)$ from $A$ to $B$ such that $(A_i, A_{i+1}) \in E$ for $i \in [k-1]$. Denote the set formed by the vertex sequence from $A$ to $B$ by $\mathscr{S}(A,B)$. The length between $A$ and $B$ is $\mathrm{l}(A,B) := \min \big\{k \in \mathbb{N}\ \big|\ (A = A_1, A_2, \dots, A_k = B) \in \mathscr{S}(A,B)\big\}$ with $\mathrm{l}(A,A) := 0$. Moreover, the set formed by the minimal sequences from $A$ to $B$ is
$$\mathscr{M}(A,B) := \big\{(A = A_1, A_2, \dots, A_k = B) \in \mathscr{S}(A,B)\ \big|\ k = \mathrm{l}(A,B)\big\}.$$

\noindent For simplicity, this article considers the model of a moving cat located in a certain room at each step. To this model can naturally be extrapolated various models. The cat goes from room $A$ to room $B$ with the probability $\mathtt{p}(A,B)$. The probabilistic graph of that model is a connected digraph $\mathrm{G} := (\mathrm{V},\mathrm{E},\mathtt{p})$ formed by the room set $\mathrm{V}$, the set $\mathrm{E} \subseteq \mathrm{V}^2$ of $2$-adjacent rooms containing also $\big\{(A_i,A_i)\big\}_{A_i \in \mathrm{V}}$, and the probability $\mathtt{p}: \mathrm{V}^2 \rightarrow [0,1]$ labeling each pair $(A_i,A_j) \in \mathrm{E}$ by $\mathtt{p}(A_i,A_j)$ such that, for $A,B \in \mathrm{V}$ with $A \neq B$,
\begin{itemize}
\item $\displaystyle \sum_{A' \in \mathrm{V}} \mathtt{p}(A,A') = 1$,
\item if $(A_1, A_2, \dots, A_k), (A_1', A_2', \dots, A_k') \in \mathscr{M}(A,B)$, then as multisets 
$$\big\{\mathtt{p}(A_i, A_{i+1})\big\}_{i \in [k-1]} = \big\{\mathtt{p}(A_i', A_{i+1}')\big\}_{i \in [k-1]},$$
\item if $(A_1, A_2, \dots, A_k) \in \mathscr{M}(A,B)$, then $\displaystyle \mathrm{p}(A,B) = \prod_{i \in [k-1]} \mathrm{p}(A_i,A_{i+1})$.
\end{itemize}

\noindent Let us call such a digraph ``A Probabilistic Graph of a Walking Cat".

\begin{definition} \label{DeCo}
Let $\mathrm{G} = (\mathrm{V},\mathrm{E},\mathtt{p})$ be a probabilistic graph of a walking cat. We say that a nonempty set $\mathrm{U} \subseteq \mathrm{V}$ is connected by a corridor if $\mathrm{V}$ can be partitioned into $\#\mathrm{U}$ sets $\mathrm{V}_1, \dots, \mathrm{V}_{\#\mathrm{U}}$ such that, for $i,j \in [\#\mathrm{U}]$,
\begin{itemize}
\item $\mathrm{V}_i \cap \mathrm{U}$ contains exactly one element which we denote $C_i$, and $(C_i,C_j) \in \mathrm{E}$,
\item if $A,B \in \mathrm{V}_i$, $A \neq B$, and $(A = A_1, A_2, \dots, A_k = B) \in \mathscr{M}(A,B)$, then $A_1, \dots, A_k \in \mathrm{V}_i$,
\item if $i \neq j$, $(A,B) \in \mathrm{V}_i \times \mathrm{V}_j$, $(A_1, A_2, \dots, A_k) \in \mathscr{M}(A,C_i)$, $(B_1, B_2, \dots, B_l) \in \mathscr{M}(C_j,B)$, then $(A_1, A_2, \dots, A_k, B_1, B_2, \dots, B_l) \in \mathscr{M}(A,B)$.
\end{itemize}
\end{definition}

\begin{example}
In the probabilistic graph of Figure~\ref{ExPr}, the set $\{2,3,4\}$ is connected by a corridor, and partitions the room set into $\{1,2\}$, $\{3,5\}$, $\{4,6\}$.	
	
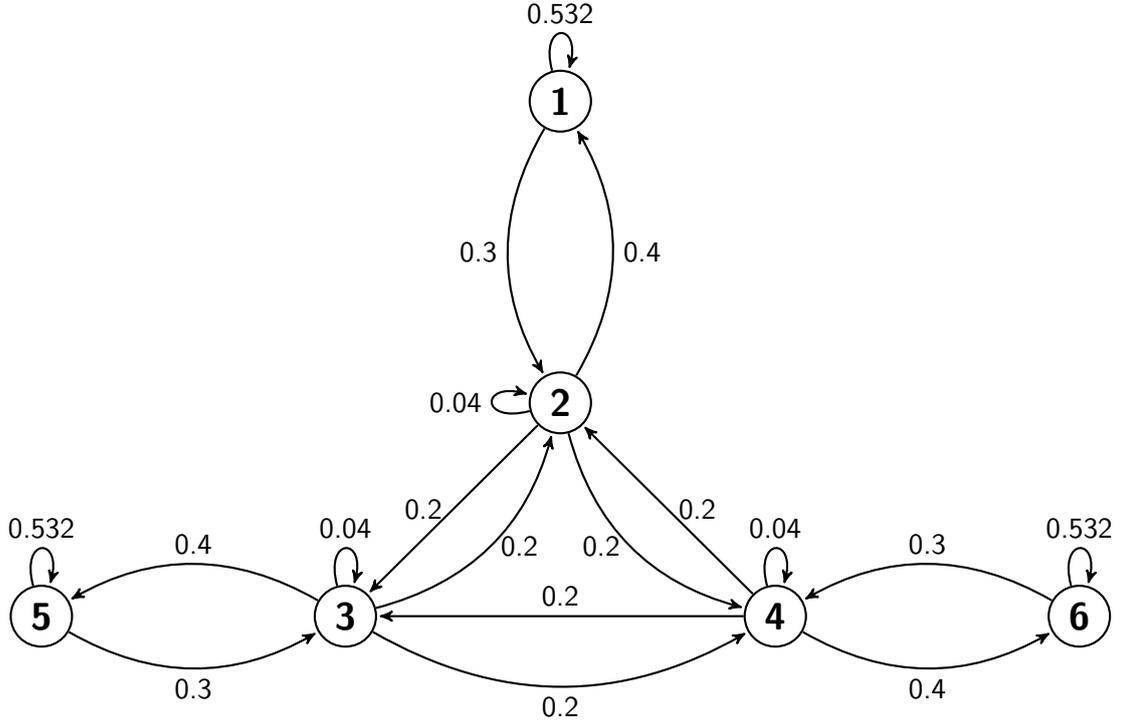
\begin{figure}[h]
	\centering
	\begin{tikzpicture}[->,>=stealth',shorten >=1pt,auto,node distance=4cm,
	thick,main node/.style={circle,draw,font=\sffamily\Large\bfseries}]
	\node[main node] (1) {1};
	\node[main node] (2) [below of=1] {2};
	\node[main node] (3) [below left of=2] {3};
	\node[main node] (4) [below right of=2] {4};
	\node[main node] (5) [left of=3] {5};
	\node[main node] (6) [right of=4] {6};
	
	\path[every node/.style={font=\sffamily}]
	(1) edge [bend right] node[left] {0.3} (2)
	edge [loop above] node {0.532} (1)
	(2) edge [bend right] node [right] {0.4} (1)
	edge [bend right] node [left] {0.2} (4)
	edge [loop left] node {0.04} (2)
	edge node[left] {0.2} (3)
	(3) edge [bend right] node [right] {0.2} (2)
	edge [loop above] node {0.04} (3)
	edge [bend right] node [below] {0.2} (4)
	edge [bend right] node [above] {0.4} (5)
	(4) edge node [right] {0.2} (2)
	edge [loop above] node {0.04} (4)
	edge node [above] {0.2} (3)
	edge [bend right] node[below] {0.4} (6)
	(5) edge [bend right] node [below] {0.3} (3)
	edge [loop above] node {0.532} (5)
	(6) edge [bend right] node [above] {0.3} (4)
	edge [loop above] node {0.532} (6);
	\end{tikzpicture}
	\caption{A Probabilistic Graph of a Walking Cat}
	\label{ExPr}
\end{figure}	
\end{example}

\noindent Before presenting the results, we need the following lemma that we prove in Section~\ref{SePr}.

\begin{lemma} \label{LeCo}
Let $\mathrm{G} = (\mathrm{V},\mathrm{E},\mathtt{p})$ be a probabilistic graph of a walking cat, and $\mathrm{U}_1, \dots, \mathrm{U}_r$ nonempty subsets of $\mathrm{V}$ with $\mathrm{U}_i \cap \mathrm{U}_j = \emptyset$ if $i,j \in [r]$ and $i \neq j$. Suppose that the sets $\mathrm{U}_1, \dots, \mathrm{U}_r$ are connected by a corridor. Then, $\mathrm{V}$ can be partitioned into $\displaystyle s = \sum_{i \in [r]} \#\mathrm{U}_i -r+1$ sets $\mathrm{V}_1, \dots, \mathrm{V}_s$ such that, for $l \in [s]$,
\begin{itemize}
\item $\exists i \in [r]:\, \mathrm{U}_i \cap \mathrm{V}_l \neq \emptyset$ \: and \: $\forall i \in [r]:\, \# \mathrm{U}_i \cap \mathrm{V}_l \leq 1$,
\item if $A,B \in \mathrm{V}_l$, $A \neq B$, and $(A = A_1, A_2, \dots, A_k = B) \in \mathscr{M}(A,B)$, then $A_1, \dots, A_k \in \mathrm{V}_l$.
\end{itemize}
\end{lemma}

\noindent The left stochastic matrix associated to a probabilistic graph of a walking cat $\mathrm{G} = (\mathrm{V},\mathrm{E},\mathtt{p})$ is $\mathsf{S}_{\mathrm{G}} := \big(\mathtt{p}(B,A)\big)_{A,B \in \mathrm{V}}$. Moreover, let $\mathrm{E}_i := \big\{(A,B) \in E\ \big|\ A,B \in \mathrm{V}_i\big\}$ for $\mathrm{V}_i$ defined in Lemma~\ref{LeCo}. It is clear that the induced subgraph $\mathrm{G}_i = (\mathrm{V}_i,\mathrm{E}_i)$ of $\mathrm{G}$ is connected. The matrix associated to that subgraph is $\mathsf{S}_{\mathrm{G}_i} := \big(\mathtt{p}(B,A)\big)_{A,B \in \mathrm{V}_i}$. We can now state the results.

\begin{theorem} \label{Th1}
Let $\mathrm{G} = (\mathrm{V},\mathrm{E},\mathtt{p})$ be a probabilistic graph of a walking cat, and $\mathrm{U}_1, \dots, \mathrm{U}_r$ nonempty subsets of $\mathrm{V}$ connected by a corridor and partitioning $\mathrm{V}$ into $s$ sets $\mathrm{V}_1, \dots, \mathrm{V}_s$ as in Lemma~\ref{LeCo}. Besides, assume that for every $i \in [r]$, there is a real number $c_i \in [0,1]$ such that, if $A,B \in \mathrm{U}_i$, we have $\mathtt{p}(A,B) = c_i$. If $\mathrm{G}_k = (\mathrm{V}_k,\mathrm{E}_k)$ are the induced subgraphs, then
$$\det \mathsf{S}_{\mathrm{G}} = \prod_{i \in [r]} \bigg(1 + \sum_{\substack{\mathrm{K} \subseteq \mathrm{U}_i \\ \#\mathrm{K} \geq 2}} (-1)^{\#\mathrm{K}-1}(\#\mathrm{K}-1) \prod_{A \in \mathrm{K}} c_i\, \mathtt{p}(A,A)\bigg) \prod_{k \in [s]} \det \mathsf{S}_{\mathrm{G}_k}.$$
\end{theorem}

\noindent Let $x_1, \dots, x_n$ be variables, and $\mathbb{M}_n$ the set formed by the monomials of $\mathbb{R}[x_1, \dots, x_n]$. Say that the function $\mathtt{d}: \mathrm{V}^2 \rightarrow \mathbb{M}_n$ is an exponential distance on a probabilistic graph of a walking cat $\mathrm{G} = (\mathrm{V},\mathrm{E},\mathtt{p})$ if, for $A,B \in \mathrm{V}$ with $A \neq B$,
\begin{itemize}
\item $\mathtt{d}(A, A)=1$,
\item if $(A_1, A_2, \dots, A_k), (A_1', A_2', \dots, A_k') \in \mathscr{M}(A,B)$, then as multisets 
$$\big\{\mathtt{d}(A_i, A_{i+1})\big\}_{i \in [k-1]} = \big\{\mathtt{d}(A_i', A_{i+1}')\big\}_{i \in [k-1]},$$
\item if $(A_1, A_2, \dots, A_k) \in \mathscr{M}(A,B)$, then $\displaystyle \mathrm{d}(A,B) = \prod_{i \in [k-1]} \mathrm{d}(A_i,A_{i+1})$.
\end{itemize}

\noindent The digraph $\bar{\mathrm{G}} = (\mathrm{V},\mathrm{E},\mathtt{d})$ is dual to the probability graph of a walking cat $\mathrm{G} = (\mathrm{V},\mathrm{E},\mathtt{p})$ if
$$\forall A,B \in \mathrm{V}:\, \mathtt{p}(A,B) = \frac{\mathtt{d}(A,B)}{\sum_{C \in \mathrm{V}} \mathtt{d}(A,C)}.$$
For $\bar{\mathrm{G}} = (\mathrm{V},\mathrm{E},\mathtt{d})$, each pair $(A,B) \in \mathrm{E}$ is labeled by $\mathtt{d}(A,B)$. Let us call such a digraph ``An Exponential Distance Graph of a Walking Cat". The distance matrix associated to an exponential distance graph of a walking cat $\bar{\mathrm{G}} = (\mathrm{V},\mathrm{E},\mathtt{d})$ is $\mathsf{D}_{\bar{\mathrm{G}}} := \big(\mathtt{d}(B,A)\big)_{A,B \in \mathrm{V}}$. Moreover, for $\mathrm{V}_i$ defined in Lemma~\ref{LeCo}, the matrix associated to the induced subgraph $\bar{\mathrm{G}}_i = (\mathrm{V}_i,\mathrm{E}_i)$ of $\bar{\mathrm{G}}$ is $\mathsf{D}_{\bar{\mathrm{G}}_i} := \big(\mathtt{d}(B,A)\big)_{A,B \in \mathrm{V}_i}$.

\begin{theorem} \label{Th2}
Let $\bar{\mathrm{G}} = (\mathrm{V},\mathrm{E},\mathtt{d})$ be an exponential distance graph of a walking cat, and $\mathrm{U}_1, \dots, \mathrm{U}_r$ nonempty subsets of $\mathrm{V}$ connected by a corridor and partitioning $\mathrm{V}$ into $s$ sets $\mathrm{V}_1, \dots, \mathrm{V}_s$ as in Lemma~\ref{LeCo}. Besides, assume that for every $i \in [r]$, there is a real variable $q_i$ such that, if $A,B \in \mathrm{U}_i$, we have $\mathtt{d}(A,B) = q_i$. If $\bar{\mathrm{G}}_k = (\mathrm{V}_k,\mathrm{E}_k)$ are the induced subgraphs, then
$$\det \mathsf{D}_{\bar{\mathrm{G}}} = \prod_{i \in [r]} \big(1 + (\#\mathrm{U}_i -1)q_i\big)(1-q_i)^{\#\mathrm{U}_i-1} \prod_{k \in [s]} \det \mathsf{D}_{\bar{\mathrm{G}}_k}.$$
\end{theorem}

\noindent This article is structured as follows: We first compute a determinant constructed from a set of square matrices in Section~\ref{SeSet}. It will be used to prove Theorem~\ref{Th1} and Theorem~\ref{Th2} in Section~\ref{SePr}. Then, we finish with the determinants of exponential distance matrices constructed from indirectly acyclic digraphs and hyperplane arrangements in Appendix~\ref{SeCo}.

\section{A Determinant based on a Set of Matrices} \label{SeSet}

\noindent We compute a determinant defined from a set of square matrices. The author would like to thank \href{https://math.cornell.edu/marcelo-aguiar}{Marcelo Aguiar} for having led him to that computing.

\begin{definition} \label{DeMat}
	Let $\mathsf{A}_1, \dots, \mathsf{A}_r$ be square matrices such that $I_k = \{i_1^k, \dots, i_{n_k}^k\}$ indexes $\mathsf{A}_k$ and $\mathsf{A}_k = (a_{i,j})_{i,j \in I_k}$ for $k \in [r]$. Define the square matrix $M_q(\mathsf{A}_1, \dots, \mathsf{A}_r) = (m_{i,j})_{i,j \in I}$ indexed by $\displaystyle I = \bigsqcup_{k \in [r]}I_k$ such that, if $i \in I_h$ and $j \in I_k$, then $\displaystyle m_{i,j} := \begin{cases}
	a_{i,j} & \text{if}\ h=k,\\
	q \cdot a_{i, i_1^h} \cdot a_{i_1^k, j} & \text{otherwise}
	\end{cases}$.
\end{definition}

\begin{example} \label{ExAB}
	If $\mathsf{A} = \begin{pmatrix} a_{11} & a_{12} \\ a_{21} & a_{22}
	\end{pmatrix}$ and $\mathsf{B} = \begin{pmatrix} b_{11} & b_{12} & b_{13} \\ b_{21} & b_{22} & b_{23} \\ b_{31} & b_{32} & b_{33}
	\end{pmatrix}$, then
	$$M_q(\mathsf{A}, \mathsf{B}) = \begin{pmatrix}
	a_{11} & a_{12} & q a_{11} b_{11} & q a_{11} b_{12} & q a_{11} b_{13} \\
	a_{21} & a_{22} & q a_{21} b_{11} & q a_{21} b_{12} & q a_{21} b_{13} \\
	q a_{11} b_{11} & q a_{12} b_{11} & b_{11} & b_{12} & b_{13} \\
	q a_{11} b_{21} & q a_{12} b_{21} & b_{21} & b_{22} & b_{23} \\
	q a_{11} b_{31} & q a_{12} b_{31} & b_{31} & b_{32} & b_{33}
	\end{pmatrix}.$$
\end{example}

\noindent Denote by $\mathfrak{D}_n$ the set formed by the derangements of order $n$. 

\begin{lemma} \label{LeMat}
	Take an integer $n \geq 2$, and $n$ variables $a_1, \dots, a_n$. Then,
	$$\det \begin{pmatrix}
	a_1 & 1 & \cdots & 1 \\
	1 & a_2 & \cdots & 1 \\
	\vdots & \vdots & \ddots & \vdots \\
	1 & 1 & \cdots & a_n
	\end{pmatrix} = \prod_{i \in [n]}a_i + \sum_{\substack{I \subseteq [n] \\ \#I \leq n-2}} (-1)^{n-\#I-1}(n-\#I-1) \prod_{i \in I}a_i.$$
\end{lemma}

\begin{proof}
	Denoting by $\Delta$ the aimed determinant, it is clear that $\displaystyle \big[\prod_{i \in [n]}a_i \big] \Delta = 1$ and for $I \subseteq [n]$ such that $\#I = n-1$ we have $\displaystyle \big[\prod_{i \in I}a_i \big] \Delta = 0$. Now if $\#I \leq n-2$, using \cite[Theorem~3.2]{Sh}, we obtain $$\big[\prod_{i \in I}a_i \big] \Delta = \prod_{\sigma \in \mathfrak{D}_{n-\#I}} \mathrm{sgn}\,\sigma = (-1)^{n-\#I-1}(n-\#I-1).$$
\end{proof}

\noindent Denote by $\mathsf{I}_n$ the identity matrix of order $n$.

\begin{theorem} \label{ThMat}
	Let $\mathsf{A}_1, \dots, \mathsf{A}_r$ be square matrices such that $I_k = \{i_1^k, \dots, i_{n_k}^k\}$ indexes $\mathsf{A}_k$ and $\mathsf{A}_k = (a_{i,j})_{i,j \in I_k}$ for $k \in [r]$. Then,
	$$\det M_q(\mathsf{A}_1, \dots, \mathsf{A}_r) =  \bigg(1 + \sum_{\substack{K \subseteq [r] \\ \#K \geq 2}} (-1)^{\#K-1}(\#K-1) \prod_{k \in K} qa_{i_1^k,i_1^k}\bigg) \prod_{k \in [r]} \det \mathsf{A}_k.$$
\end{theorem}

\begin{proof}
	Remark first that $M_q(\mathsf{A}_1, \dots, \mathsf{A}_r)$ is equal to the product of the square matrix $\displaystyle \bigoplus_{k \in [r]}^{\rightarrow} \mathsf{A}_k$ with the square matrix $\displaystyle \mathsf{F} = \begin{pmatrix}
	\mathsf{F}_{11} & \mathsf{F}_{12} & \cdots & \mathsf{F}_{1r} \\
	\mathsf{F}_{21} & \mathsf{F}_{22} & \cdots & \mathsf{F}_{2r} \\
	\vdots & \vdots & \ddots & \vdots \\
	\mathsf{F}_{r1} & \mathsf{F}_{r2} & \cdots & \mathsf{F}_{rr}
	\end{pmatrix}$ where $\mathsf{F}_{hk}$ is the $n_k \times n_h$ matrix such that
	$$\mathsf{F}_{hk} = \begin{cases}
	\quad \mathsf{I}_{n_h} & \text{if}\ h=k, \\
	\begin{pmatrix}
	qa_{i_1^k,i_1^k} & qa_{i_1^k,i_2^k} & qa_{i_1^k,i_3^k} & \dots & qa_{i_1^k,i_{n_k}^k} \\
	0 & 0 & 0 & \cdots & 0 \\
	\vdots & \vdots & \vdots & \ddots & \vdots \\
	0 & 0 & 0 & \cdots & 0
	\end{pmatrix} & \text{otherwise}.
	\end{cases}$$
	In Example~\ref{ExAB}, $\displaystyle M_q(\mathsf{A}, \mathsf{B}) = \begin{pmatrix}
	a_{11} & a_{12} & 0 & 0 & 0 \\
	a_{21} & a_{22} & 0 & 0 & 0 \\
	0 & 0 & b_{11} & b_{12} & b_{13} \\
	0 & 0 & b_{21} & b_{22} & b_{23} \\
	0 & 0 & b_{31} & b_{32} & b_{33}
	\end{pmatrix} \begin{pmatrix}
	1 & 0 & q b_{11} & q b_{12} & q b_{13} \\
	0 & 1 & 0 & 0 & 0 \\
	q a_{11} & q a_{12} & 1 & 0 & 0 \\
	0 & 0 & 0 & 1 & 0 \\
	0 & 0 & 0 & 0 & 1
	\end{pmatrix}$ for instance. Let $\displaystyle I = \bigsqcup_{k \in [r]}I_k$ and $\displaystyle J = \bigsqcup_{k \in [r]}\{i_1^k\}$. Using the determinantal formula, we obtain $$\det \mathsf{F} = \det \mathsf{F}[J] \ \det\big(\mathsf{F}[I \setminus J] - \mathsf{F}[I \setminus J, J] \mathsf{F}[J]^{-1} \mathsf{F}[J, I \setminus J] \big)$$
	where $\mathsf{F}[J]$ is the $r \times r$ circulant matrix $\begin{pmatrix}
	1 & qa_{i_1^2,i_1^2} & \cdots & qa_{i_1^r,i_1^r} \\ qa_{i_1^1,i_1^1} & 1 & \cdots & qa_{i_1^r,i_1^r} \\
	\vdots & \ddots & \ddots & \vdots \\
	qa_{i_1^1,i_1^1} & \cdots & qa_{i_1^{r-1},i_1^{r-1}} & 1
	\end{pmatrix}$, $\mathsf{F}[I \setminus J] = \mathsf{I}_{\#I \setminus J}$, and $\mathsf{F}[I \setminus J, J]$ is the $\#I \setminus J \times r$ null matrix. Using Lemma~\ref{LeMat}, we obtain
	$$\det \mathsf{F}[J] = \prod_{k \in [r]} qa_{i_1^k,i_1^k} \times \begin{vmatrix}
	\frac{1}{qa_{i_1^1,i_1^1}} & 1 & \cdots & 1 \\
	1 & \frac{1}{qa_{i_1^2,i_1^2}} & \cdots & 1 \\
	\vdots & \ddots & \ddots & \vdots \\
	1 & \cdots & 1 & \frac{1}{qa_{i_1^r,i_1^r}}
	\end{vmatrix} = 1 + \sum_{\substack{K \subseteq [r] \\ \#K \geq 2}} (-1)^{\#K-1}(\#K-1) \prod_{k \in K} qa_{i_1^k,i_1^k}.$$
	Finally with $\displaystyle \det \bigoplus_{k \in [r]}^{\rightarrow} \mathsf{A}_k = \prod_{k \in [r]} \det \mathsf{A}_k$, we get the result.
\end{proof}

\section{Proof of Lemma~\ref{LeCo}, Theorem~\ref{Th1}, and Theorem~\ref{Th2}} \label{SePr}

\noindent We begin by proving Lemma~\ref{LeCo}, then Theorem~\ref{Th1}, and finally Theorem~\ref{Th2}.

\begin{proof}
	Consider first $\mathrm{U}_1$ partitioning $\mathrm{V}$ into $\mathrm{V}_1^{(1)}, \dots, \mathrm{V}_{\#\mathrm{U}_1}^{(1)}$. For $i \in [\#\mathrm{U}_1]$, set $\mathrm{U}_1 \cap \mathrm{V}_i^{(1)} = \{C_i^{(1)}\}$. If $i,j \in [\#\mathrm{U}_1]$ with $i \neq j$, as $\big\{(A,B)\ \big|\ A \in \mathrm{V}_i^{(1)},\, B \in \mathrm{V}_j^{(1)}\big\} = \big\{(C_i^{(1)}, C_j^{(1)})\big\}$, $\mathrm{U}_2$ is then included in some $\mathrm{V}_i^{(1)}$ that we assume to be $\mathrm{V}_{\#\mathrm{U}_1}^{(1)}$. From its definition, $\mathrm{U}_2$ also partitions $\mathrm{V}_{\#\mathrm{U}_1}^{(1)}$ into $\mathrm{V}_1^{(2)}, \dots, \mathrm{V}_{\#\mathrm{U}_2}^{(2)}$, and the partition $\mathrm{V}_1^{(1)}, \dots, \mathrm{V}_{\#\mathrm{U}_1 -1}^{(1)}, \mathrm{V}_1^{(2)}, \dots, \mathrm{V}_{\#\mathrm{U}_2}^{(2)}$ has the property of Lemma~\ref{LeCo} for $\mathrm{U}_1$ and $\mathrm{U}_2$. By induction, we obtain the partition of $\displaystyle \sum_{i \in [r-1]} \#\mathrm{U}_i -r+2$ sets $\mathrm{V}_1^{(1)}, \dots, \mathrm{V}_{\#\mathrm{U}_1 -1}^{(1)}, \mathrm{V}_1^{(2)}, \dots, \mathrm{V}_{\#\mathrm{U}_2-1}^{(2)}, \dots, \mathrm{V}_1^{(n-1)}, \dots, \mathrm{V}_{\#\mathrm{U}_{n-1}}^{(n-1)}$ having the property of Lemma~\ref{LeCo} after the $(n-1)^{\text{th}}$ step. Taking any two different sets $\mathrm{V}_i^{(l)}, \mathrm{V}_j^{(k)}$ of those latter, either $\big\{(A,B)\ \big|\ A \in \mathrm{V}_i^{(l)},\, B \in \mathrm{V}_j^{(k)}\big\}$ is equal to some $\big\{(C_i^{(l)}, C_m^{(l)})\big\}$ or is empty. Hence, $\mathrm{U}_n$ is included in exactly one of these $\displaystyle \sum_{i \in [r-1]} \#\mathrm{U}_i -r+2$ sets that we assume to be $\mathrm{V}_{\#\mathrm{U}_{n-1}}^{(n-1)}$. After its partitioning by $\mathrm{U}_n$, we finally obtain the desired $\displaystyle \sum_{i \in [r]} \#\mathrm{U}_i -r+1$ sets $\mathrm{V}_1^{(1)}, \dots, \mathrm{V}_{\#\mathrm{U}_1 -1}^{(1)}, \mathrm{V}_1^{(2)}, \dots, \mathrm{V}_{\#\mathrm{U}_2-1}^{(2)}, \dots, \mathrm{V}_1^{(n-1)}, \dots, \mathrm{V}_{\#\mathrm{U}_{n-1}-1}^{(n-1)}, \mathrm{V}_1^{(n)}, \dots, \mathrm{V}_{\#\mathrm{U}_{n}}^{(n)}$.
\end{proof}

\begin{proof}
Considering the sets $\mathrm{V}_1^{(1)}, \dots, \mathrm{V}_{\#\mathrm{U}_1}^{(1)}, \dots, \mathrm{V}_1^{(n)}, \dots, \mathrm{V}_{\#\mathrm{U}_{n}}^{(n)}$ in the proof of Lemma~\ref{LeCo}, let $\mathsf{S}_{\mathrm{G}_i^{(k)}} := \big(\mathtt{p}(B,A)\big)_{A,B \in \mathrm{V}_i^{(k)}}$. Using Theorem~\ref{ThMat}, we successively get
\begin{align*}
\det \mathsf{S}_{\mathrm{G}} & = \det M_{c_1}(\mathrm{V}_1^{(1)}, \dots, \mathrm{V}_{\#\mathrm{U}_1}^{(1)}) \\
& = \bigg(1 + \sum_{\substack{\mathrm{K} \subseteq \mathrm{U}_1 \\ \#\mathrm{K} \geq 2}} (-1)^{\#\mathrm{K}-1}(\#\mathrm{K}-1) \prod_{A \in \mathrm{K}} c_1\, \mathtt{p}(A,A)\bigg) \prod_{k \in [\#\mathrm{U}_1-1]} \det \mathsf{S}_{\mathrm{G}_k^{(1)}} \\
& \quad \times \det M_{c_2}(\mathrm{V}_1^{(2)}, \dots, \mathrm{V}_{\#\mathrm{U}_2}^{(2)}) \\
& = \prod_{i \in [r]} \bigg(1 + \sum_{\substack{\mathrm{K} \subseteq \mathrm{U}_i \\ \#\mathrm{K} \geq 2}} (-1)^{\#\mathrm{K}-1}(\#\mathrm{K}-1) \prod_{A \in \mathrm{K}} c_i\, \mathtt{p}(A,A)\bigg) \\
& \quad \times \prod_{l \in [n]} \prod_{k \in [\#\mathrm{U}_l-1]} \det \mathsf{S}_{\mathrm{G}_k^{(l)}} \times \det \mathsf{S}_{\mathrm{G}_{\#\mathrm{U}_n}^{(n)}}.
\end{align*}
\end{proof}

\begin{proof}
With an argument similar to the proof of Theorem~\ref{Th1}, we obtain
$$\det \mathsf{D}_{\bar{\mathrm{G}}} = \prod_{i \in [r]} \bigg(1 + \sum_{\substack{\mathrm{K} \subseteq \mathrm{U}_i \\ \#\mathrm{K} \geq 2}} (-1)^{\#\mathrm{K}-1}(\#\mathrm{K}-1) \prod_{A \in \mathrm{K}} q_i\, \mathtt{d}(A,A)\bigg) \times \prod_{l \in [n]} \prod_{k \in [\#\mathrm{U}_l-1]} \det \mathsf{D}_{\bar{\mathrm{G}}_k^{(l)}} \times \det \mathsf{D}_{\bar{\mathrm{G}}_{\#\mathrm{U}_n}^{(n)}}.$$ Then, on one side $\mathtt{d}(A,A) = 1$, and on the other side
	\begin{align*}
	\big(1 + (\#\mathrm{U}_i -1)q_i\big)(1-q_i)^{\#\mathrm{U}_i-1} & = \big(1 + (\#\mathrm{U}_i-1)q_i\big) \sum_{k=0}^{\#\mathrm{U}_i-1}(-1)^k \binom{\#\mathrm{U}_i-1}{k} q_i^k \\
	& = 1 + (-1)^{\#\mathrm{U}_i-1}(\#\mathrm{U}_i-1)q_i^{\#\mathrm{U}_i} \\
	& \quad + \sum_{k=0}^{\#\mathrm{U}_i-2} \bigg((-1)^k (\#\mathrm{U}_i-1) \binom{\#\mathrm{U}_i-1}{k} + (-1)^{k+1} \binom{\#\mathrm{U}_i-1}{k+1}\bigg)q_i^{k+1} \\
	& = 1 + (-1)^{\#\mathrm{U}_i-1}(r-1)q_i^{\#\mathrm{U}_i} + \sum_{k=0}^{\#\mathrm{U}_i-2} (-1)^k k \binom{\#\mathrm{U}_i}{k+1} q_i^{k+1} \\
	& = 1 + \sum_{k=2}^{\#\mathrm{U}_i} (-1)^{k-1}(k-1) \binom{\#\mathrm{U}_i}{k} q_i^k.
	\end{align*}
\end{proof}

\appendix

\section{Examples of Exponential Distance} \label{SeCo}

\noindent We compute the determinant of matrices associated to two exponential distance graphs.

\paragraph{Indirectly Acyclic Digraph.} Transform a digraph $\mathrm{G} = (\mathrm{V},\mathrm{E})$ to an undirected graph $\mathrm{u(G)} = (\mathrm{V},\mathrm{u(E)})$ by defining $\displaystyle \mathrm{u(E)} := \big\{\{A,B\} \in \binom{\mathrm{V}}{2}\ \big|\ (A,B) \in \mathrm{E}\big\}$. We say that the digraph $\mathrm{G}$ is indirectly acyclic if the undirected graph $\mathrm{u(G)}$ is acyclic.

\begin{lemma} \label{LeDi}
Let $\bar{\mathrm{G}} = (\mathrm{V},\mathrm{E},\mathtt{d})$ be an indirectly acyclic exponential distance graph of a walking cat. Then,
$$\det \mathsf{D}_{\bar{\mathrm{G}}} = \prod_{\{A,B\} \in \mathrm{u(E)}} \big(1 - \mathtt{d}(A,B)\, \mathtt{d}(B,A)\big).$$
\end{lemma}

\begin{proof}
We proceed by induction on the number of rooms. Assume $\mathrm{V} = \{A_1, \dots, A_n\}$, and Lemma~\ref{LeDi} for $\bar{\mathrm{G}} = (\mathrm{V},\mathrm{E},\mathtt{d})$. Then, consider the extension $\bar{\mathrm{G}}' = (\mathrm{V}',\mathrm{E}',\mathtt{d}')$ of $\bar{\mathrm{G}}$ such that $\mathrm{V}' = \mathrm{V} \sqcup \{B\}$, $\mathrm{E}' = \mathrm{E} \sqcup \big\{(A_n,B), (B,A_n)$, and $\mathtt{d}'(A_i,A_j) = \mathtt{d}(A_i,A_j)$ for $i,j \in [n]$. Hence,
\begin{align*}
\det \mathsf{D}_{\bar{\mathrm{G}}'} & = 
\begin{vmatrix} 1 & \mathtt{d}(A_2,A_1) & \cdots & \mathtt{d}(A_n,A_1) & \mathtt{d}'(B,A_n)\, \mathtt{d}(A_n,A_1) \\
\mathtt{d}(A_1,A_2) & 1 & \cdots & \mathtt{d}(A_n,A_2) & \mathtt{d}'(B,A_n)\, \mathtt{d}(A_n,A_2) \\
\vdots & \vdots & \ddots & \vdots & \vdots \\
\mathtt{d}(A_1,A_n) & \mathtt{d}(A_2,A_n) & \cdots & 1 & \mathtt{d}'(B,A_n) \\
\mathtt{d}'(A_1,B) & \mathtt{d}'(A_2,B) & \cdots & \mathtt{d}'(A_n,B) & 1
\end{vmatrix} \\
& = \begin{vmatrix} 1 & \mathtt{d}(A_2,A_1) & \cdots & \mathtt{d}(A_n,A_1) & 0 \\
\mathtt{d}(A_1,A_2) & 1 & \cdots & \mathtt{d}(A_n,A_2) & 0 \\
\vdots & \vdots & \ddots & \vdots & \vdots \\
\mathtt{d}(A_1,A_n) & \mathtt{d}(A_2,A_n) & \cdots & 1 & 0 \\
\mathtt{d}'(A_1,B) & \mathtt{d}'(A_2,B) & \cdots & \mathtt{d}'(A_n,B) & 1 - \mathtt{d}'(B,A_n) \, \mathtt{d}'(A_n,B)
\end{vmatrix} \\
& = \det \mathsf{D}_{\bar{\mathrm{G}}} \times \big(1 - \mathtt{d}'(B,A_n) \, \mathtt{d}'(A_n,B)\big). 
\end{align*}
\end{proof}

\begin{example}
The determinant of the matrix associated to the exponential distance graph of a walking cat represented in Figure~\ref{ExDi} is	
$$\begin{vmatrix} 1 & a^- & a^- b^+ & a^- c^+ & a^- d^+ & a^- d^+ e^+ \\
a^+ & 1 & b^+ & c^+ & d^+ & d^+ e^+ \\ a^+ b^- & b^- & 1 & b^- c^+ & b^- d^+ & b^- d^+ e^+ \\
a^+ c^- & c^- & b^+ c^- & 1 & c^- d^+ & c^- d^+ e^+ \\ a^+ d^- & d^- & b^+ d^- & c^+ d^- & 1 & e^+ \\
a^+ d^- e^- & d^- e^- & b^+ d^- e^- & c^+ d^- e^- & e^- & 1
\end{vmatrix} = \begin{matrix}
(1 - a^+ a^-)\, (1 - b^+ b^-)\, (1 - c^+ c^-) \\ (1 - d^+ d^-)\ (1 - e^+ e^-).
\end{matrix}$$
	
\begin{figure}[h]
	\centering
	\begin{tikzpicture}[->,>=stealth',shorten >=1pt,auto,node distance=2.5cm,
	thick,main node/.style={circle,draw,font=\sffamily\bfseries}]
	\node[main node] (1) {1};
	\node[main node] (2) [right of=1] {2};
	\node[main node] (3) [right of=2] {3};
	\node[main node] (4) [above of=2] {4};
	\node[main node] (5) [below of=2] {5};
	\node[main node] (6) [below of=5] {6};
	
	\path[every node/.style={font=\sffamily}]
	(1) edge [bend right] node [below] {$a^+$} (2)
	edge [loop left] node {1} (1)	
	(2) edge [bend right] node [above] {$a^-$} (1)
	edge [bend right] node [below] {$b^-$} (3)
	edge [bend right] node [right] {$c^-$} (4)
	edge [bend right] node [left] {$d^-$} (5)
	edge [loop above] node {1} (2)
	(3) edge [bend right] node [above] {$b^+$} (2)
	edge [loop right] node {1} (3)
	(4) edge [bend right] node [left] {$c^+$} (2)
    edge [loop above] node {1} (4)	
	(5) edge [bend right] node [right] {$d^+$} (2)
	edge [bend right] node [left] {$e^-$} (6)
	edge [loop above] node {1} (5)
	(6) edge [bend right] node [right] {$e^+$} (5)
	edge [loop below] node {1} (6);
	\end{tikzpicture}
	\caption{An Exponential Distance Graph}
	\label{ExDi}
\end{figure}
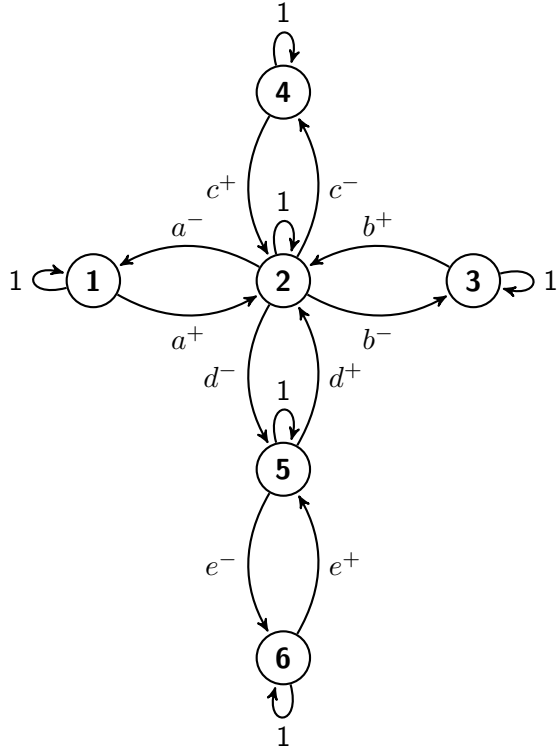	
\end{example}

\begin{proposition} 
Let $\bar{\mathrm{G}} = (\mathrm{V},\mathrm{E},\mathtt{d})$ be an exponential distance graph of a walking cat, and $\mathrm{U}_1, \dots, \mathrm{U}_r$ nonempty subsets of $\mathrm{V}$ connected by a corridor and partitioning $\mathrm{V}$ into $s$ sets $\mathrm{V}_1, \dots, \mathrm{V}_s$ as in Lemma~\ref{LeCo}. Besides, assume that
\begin{itemize}
\item for every $i \in [r]$, there is a real variable $q_i$ such that, if $A,B \in \mathrm{U}_i$, we have $\mathtt{d}(A,B) = q_i$,
\item for every $k \in [s]$, the induced subgraph $\bar{\mathrm{G}}_k = (\mathrm{V}_k,\mathrm{E}_k, \mathtt{d})$ is an indirectly acyclic digraph.
\end{itemize}	
Then, $$\det \mathsf{D}_{\bar{\mathrm{G}}} = \prod_{i \in [r]} \big(1 + (\#\mathrm{U}_i -1)q_i\big)(1-q_i)^{\#\mathrm{U}_i-1} \prod_{k \in [s]} \prod_{\{A,B\} \in \mathrm{u}(\mathrm{E}_k)} \big(1 - \mathtt{d}(A,B)\, \mathtt{d}(B,A)\big).$$
\end{proposition}

\begin{proof}
Use Theorem~\ref{Th2} and Lemma~\ref{LeDi}.
\end{proof}

\paragraph{Hyperplane Arrangement.} To every hyperplane $H$ in $\mathbb{R}^n$ can be associated two connected open half-spaces $H^+$ and $H^-$ such that $H^+ \sqcup H^0 \sqcup H^- = \mathbb{R}^n$ and $\overline{H^+} \cap \overline{H^-} = H^0$, letting $H^0 := H$. A face of a hyperplane arrangement $\mathcal{A}$ is a nonempty subset of $\mathbb{R}^n$ having the form $\displaystyle F := \bigcap_{H \in \mathcal{A}} H^{\epsilon_H(F)}$ with $\epsilon_H(F) \in \{+,0,-\}$. Denote the set formed by the faces of $\mathcal{A}$ by $F_{\mathcal{A}}$. A chamber of $\mathcal{A}$ is a face $F \in F_{\mathcal{A}}$ such that $\epsilon_H(F) \neq 0$ for every $H \in \mathcal{A}$. Denote the set formed by the chambers of $\mathcal{A}$ by $C_{\mathcal{A}}$. For $A,B \in C_{\mathcal{A}}$, the set of half-spaces containing $A$ but not $B$ is $\mathscr{H}(A,B) := \big\{H^{\epsilon_H(A)}\ \big|\ H \in \mathcal{A},\, \epsilon_H(A)=- \epsilon_H(B)\big\}$. Assign a variable $h_H^{\varepsilon}$ to every half-space $H^{\varepsilon}$, and define the polynomial ring $R_{\mathcal{A}} := \mathbb{R}\big[h_H^{\varepsilon}\ \big|\ \varepsilon \in \{+,-\},\, H \in \mathcal{A}\big]$. The exponential distance $\mathrm{v}: C_{\mathcal{A}} \times C_{\mathcal{A}} \rightarrow R_{\mathcal{A}}$ of Aguiar and Mahajan is \cite[§~8.1]{AgMa} $$\mathrm{v}(A,A) = 1 \quad \text{and} \quad \mathrm{v}(A,B) = \prod_{H^{\varepsilon} \in \mathscr{H}(A,B)} h_H^{\varepsilon}\,\ \text{if}\,\ A \neq B.$$

\noindent The centralization to a face $F \in F_{\mathcal{A}} \setminus C_{\mathcal{A}}$ is defined by $\mathcal{A}_F := \{H \in \mathcal{A}\ |\ F \subseteq H\}$, its weight  
$\displaystyle \mathrm{b}_F := \prod_{H \in \mathcal{A}_F} h_H^+ \, h_H^-$, and its multiplicity $\displaystyle \beta_F := \frac{\#\{C \in C_{\mathcal{A}}\ |\ \overline{C} \cap H = F\}}{2}$ which is independent of the chosen $H \in \mathcal{A}_F$ as can be seen in \cite[Theorem~5.7]{Ra}.

\begin{proposition} 
Let $\bar{\mathrm{G}} = (\mathrm{V},\mathrm{E},\mathtt{d})$ be an exponential distance graph of a walking cat, and $\mathrm{U}_1, \dots, \mathrm{U}_r$ nonempty subsets of $\mathrm{V}$ connected by a corridor and partitioning $\mathrm{V}$ into $s$ sets $\mathrm{V}_1, \dots, \mathrm{V}_s$ as in Lemma~\ref{LeCo}. Besides, assume that
\begin{itemize}
	\item for every $i \in [r]$, there is a real variable $q_i$ such that, if $A,B \in \mathrm{U}_i$, we have $\mathtt{d}(A,B) = q_i$,
	\item for every $k \in [s]$, there exists a hyperplane arrangement $\mathcal{A}_k$ such that, if $\bar{\mathrm{G}}_k = (\mathrm{V}_k,\mathrm{E}_k, \mathtt{d})$ is the subgraph induced by $\mathrm{V}_k$, then $\mathrm{V}_k = C_{\mathcal{A}_k}$, $\mathrm{E}_k = \big\{(A,B) \in C_{\mathcal{A}_k}^2\ \big|\ \#\mathscr{H}(A,B)=1\big\}$, and $\mathtt{d}(A,B) = \mathrm{v}(A,B)$ for $A,B \in \mathrm{V}_k$.
\end{itemize}	
We have $$\det \mathsf{D}_{\bar{\mathrm{G}}} = \prod_{i \in [r]} \big(1 + (\#\mathrm{U}_i -1)q_i\big)(1-q_i)^{\#\mathrm{U}_i-1} \prod_{k \in [s]} \prod_{F \in F_{\mathcal{A}_k} \setminus C_{\mathcal{A}_k}} (1 - \mathrm{b}_F)^{\beta_F}.$$
\end{proposition}

\begin{proof}
Use Theorem~\ref{Th2} and \cite[Corollary~1.4]{Ra}.
\end{proof}

\begin{example}
Consider the exponential distance graph with induced subgraphs and entrance rooms respectively represented by the hyperplane arrangements $\mathcal{A}_1, \mathcal{A}_2$ and the set $\{C_1, C_2\}$ in Figure~\ref{ex}. In order to have a determinant calculable with SageMath, we assume that, for $i \in [3]$, $h_{H_i}^+ = h_{H_i}^- = h_1$, and $h_{H}^+ = h_{H}^- = h_2$. Moreover, set $\mathtt{d}(C_1,C_2) = \mathtt{d}(C_2,C_1) = q$. The determinant of the matrix associated to that exponential distance graph is $$\begin{vmatrix}
	1 & h_1^2 & h_1 & h_1^2 & h_1 & h_1 & h_1^2 & q & q h_2 \\
	h_1^2 & 1 & h_1 & h_1^2 & h_1 & h_1^3 & h_1^2 & h_1^2 q & h_1^2 q h_2 \\
	h_1 & h_1 & 1 & h_1 & h_1^2 & h_1^2 & h_1^3 & h_1 q & h_1 q h_2 \\
	h_1^2 & h_1^2 & h_1 & 1 & h_1^3 & h_1 & h_1^2 & h_1^2 q & h_1^2 q h_2 \\
	h_1 & h_1 & h_1^2 & h_1^3 & 1 & h_1^2 & h_1 & h_1 q & h_1 q h_2 \\
	h_1 & h_1^3 & h_1^2 & h_1 & h_1^2 & 1 & h_1 & h_1 q & h_1 q h_2 \\
	h_1^2 & h_1^2 & h_1^3 & h_1^2 & h_1 & h_1 & 1 & h_1^2 q & h_1^2 q h_2 \\
	q & h_1^2 q & h_1 q & h_1^2 q & h_1 q & h_1 q & h_1^2 q & 1 & h_2 \\
	q h_2 & h_1^2 q h_2 & h_1 q h_2 & h_1^2 q h_2 & h_1 q h_2 & h_1 q h_2 & h_1^2 q h_2 & h_2 & 1
	\end{vmatrix} = (1-q^2) (1- h_1^2)^9 (1- h_2^2).$$
	
	\begin{figure}[h]
		\centering
		\includegraphics[scale=0.7]{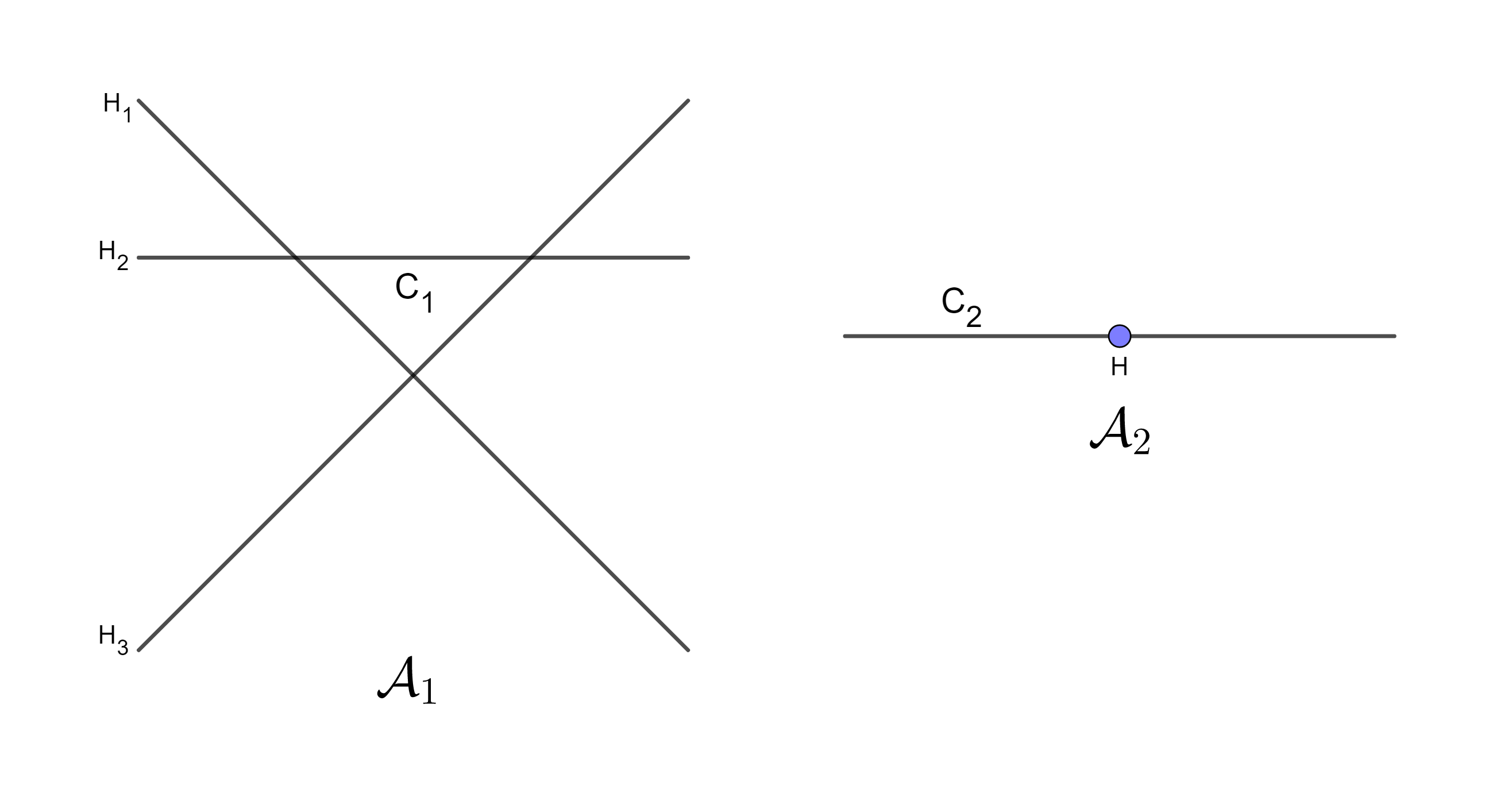}
		\caption{Hyperplane arrangements and Entrance Rooms}
		\label{ex}
	\end{figure}
\end{example}

\bibliographystyle{abbrvnat}

\end{document}